\documentclass[reqno,12pt]{amsart}
\usepackage{amsmath, amsthm, amssymb,geometry,xcolor,hyperref,enumerate}
\usepackage{verbatim}
\renewcommand{\d}{\partial}
\newcommand{\dbar}{\overline{\partial}}
\newcommand{\ddbar}{\sqrt{-1}\d\overline{\d}}
\newcommand{\ddbarb}{\sqrt{-1}\d_b\overline{\d}_b}

\newtheorem{thm}{Theorem}
\newtheorem{prop}[thm]{Proposition}
\newtheorem{lem}[thm]{Lemma}
\newtheorem{cor}[thm]{Corollary}

\newtheorem*{cl}{Claim}

\newtheorem{rem}[thm]{Remark}

\theoremstyle{definition}
\newtheorem{defn}[thm]{Definition}

\renewcommand{\[}{\begin{equation}}
	\renewcommand{\]}{\end{equation}}

\newcommand{\al}{\alpha}
\newcommand{\be}{\beta}
\newcommand{\ga}{\gamma}
\newcommand{\la}{\lambda}
\newcommand{\ka}{\kappa}
\newcommand{\ep}{\epsilon}

\newcommand{\Ga}{\Gamma}
\newcommand{\La}{\Lambda}

\newcommand{\vp}{\varphi}
\newcommand{\vep}{\varepsilon}

\newcommand{\NN}{\mathbb{N}}

\newcommand{\QQ}{\mathbb{Q}}
\newcommand{\RR}{\mathbb{R}}
\newcommand{\CC}{\mathbb{C}}

\newcommand{\PP}{\mathbb{P}}

\newcommand{\Nu}{\mathcal{V}}

\newcommand{\sL}{\mathcal{L}}

\newcommand{\sF}{\mathcal{F}}
\newcommand{\TT}{\mathbb{T}}

\numberwithin{equation}{section}
\counterwithin{thm}{section}
\counterwithin*{case}{thm}

\hypersetup{
	colorlinks=true,
	linkcolor=blue,
	citecolor=red,
	urlcolor=magenta
}
\allowdisplaybreaks
\author{P. Sivaram}
\address{Department of Mathematics, Indian Institute of Science, Bangalore-560012, India.}
\email{sivaramp@iisc.ac.in}
\begin{document}
	\title[$L^\infty$ estimate for transverse Complex Monge-Amp\`ere equation and application]{A note on the PDE approach to the $L^\infty$ estimates for complex Hessian equations on transverse K\"ahler manifolds}
	\begin{abstract}
	We adapt the PDE approach of Guo-Phong-Tong and Guo-Phong-Tong-Wang \cite{GPT21,GPTW21} to prove an  $L^\infty$ estimate for transverse complex Monge-Amp\`ere equations on homologically orientable transverse K\"ahler manifolds. As an application, we obtain a purely PDE proof of the regularity of Calabi-Yau cone metrics on $\mathbb{Q}$-Gorenstein $\mathbb{T}$-varieties. 
	\end{abstract}
	\maketitle
	
\section{Introduction}
	The celebrated work of Yau \cite{Yau} on the Calabi conjecture sparked a surge of interest in various aspects of the complex Monge-Ampère equation and its applications. Degenerate Monge-Amp\`ere equations, which already appear in Yau's original work, are unavoidable when dealing with singular varieties or cohomology classes that may not be K\"ahler (cf. \cite{BBEGZ,FGS}). We refer the reader to the survey article \cite{sps} and the references therein for more on complex Monge-Amp\`ere equation on compact K\"ahler manifolds.
	
	The $L^\infty$ estimate was the hardest part in solving the complex Monge-Amp\`ere equation and it was achieved in \cite{Yau} by using the Moser iteration argument. This argument works if the right hand-side is in $L^p$ with $p>n$, where $n$ is the complex dimension of the manifold. In an important breakthrough,  Kolodziej \cite{Ko} obtained the $L^\infty$ estimate, assuming that the right hand-side is in $L^1(\log L)^p$ with $p>1$, and his proof relied completely on techniques from pluripotential theory. Kolodziej's methods were later extended to the singular and degenerate settings (cf. \cite{EGZ1, EGZ2}). The pluri-potential theoretic methods were recently extended to the Sasakian setting by Tran-Trung Nghiem \cite{TTR}. For a long time a purely PDE approach to the sharp $L^\infty$ estimates was lacking. The gap was finally filled in a series of works by Guo-Phong-Tong and Guo-Phong-Tong-Wang  \cite{GPT21,GPTW21}. The main goal of this note is to extend the results in \cite{GPT21,GPTW21} to degenerate transverse complex Monge-Amp\`ere equations on homologically orientable transverse K\"ahler manifolds, which includes Sasakian manifolds as a special case.
	
	Before we state our main theorem let us introduce some definitions on transverse K\"ahler manifolds. 
	The setting is as follows. The reader can refer to the section below for more details. Let $(\tilde{S},\tilde{\sF})$ be a homologically orientable transverse  K\"ahler manifold with a foliation $\tilde{\sF}$ of dimension one. 
	For $ t \in (0,1], $ we define $$ \omega_t=\omega+t\chi ~\text{ and } \omega_{t,\vp}=\omega_t+\ddbarb \vp, \text{ for }\vp \in C^2_b(\tilde{S}),$$
	where $ \omega $ is a closed basic $ (1,1) $ form which is \textit{nef} and $ \chi $ is the given transverse K\"ahler form. Then by definition $[\omega_t]$ is a transverse K\"ahler class, but $\omega_t$ may not be a transverse K\"ahler form. Since the foliation $\tilde{\sF}$ is  homologically orientable, by \cite{Masa,mol} there exists a nowhere vanishing vector field $\xi$ tangent to $\tilde{\sF}$ at all points and a Riemannian metric $g$ such that $\xi$ is a Killing field with respect to $g$. And we define a $1$-form $\eta$ on $\tilde{S}$ by $$\eta(\cdot):=\frac{g(\xi,\cdot)}{g(\xi,\xi)}.$$ Then by definition $\eta(\xi)=1$ and $d\eta$ is a basic form.
	
	Our goal is to obtain an $ L^\infty $-estimate for the solutions $\vp_t$ that is independent of $ t\in (0,1] $ of the following transverse complex Monge-Amp\`ere equations \[\label{tcma nef}\begin{cases} (\omega_t+\ddbarb\vp_t)^n\wedge\eta =c_te^F\chi^n\wedge\eta, \\ \sup_{\tilde{S}}\vp_t=0,\end{cases} \] where $$c_t=\frac{V_t}{\int_{\tilde{S}}\chi^n\wedge\eta},~V_t=\int_{\tilde{S}}
	\omega_t^n\wedge\eta$$ are positive constants, and $ F $  is a real, smooth, basic function,  normalized by $$ \int_{\tilde{S}}e^F\chi^n\wedge\eta=\int_{\tilde{S}}\chi^n \wedge\eta.$$
	Since $d\eta$ is basic, $V_t$ is independent of the representative chosen from the class $[\omega_t]$. Now we state our main result. 
\begin{thm}\label{mt}
	Let $(\tilde{S},\tilde{\sF},\chi)$ be a transverse K\"ahler manifold which is homologically orientable and $\omega$ be a basic nef class. For a fixed $p>n,$ there exist a positive constant $C=C(n,p,\xi,\omega,\chi,\|e^F\|_{L^1(\log L)^p})$ such that for all $t\in (0,1]$, we have $$0\leq -\vp_t+\Nu_t\leq C, $$ where $\Nu_t$ is the upper envelope of $\omega_t$ defined by   \[\label{nef envelope} \Nu_t(x):=\sup\{v(x):v\in PSH(\tilde{S},\omega_t,\xi),v\leq 0\}.  \]
\end{thm}
	Note that proof for the above theorem extends easily to more general complex Hessian equations on transverse K\"ahler manifolds of the kind considered in  \cite{GPT21,GPTW21}. We outline the argument in the Remark \ref{rem:more-gen-comm-hessian}.  
\begin{rem}
	There is nothing special about the assumption on the dimension of the foliation $\tilde{\sF}$, except it provides an application as the Corollary \ref{main thm}. Once the foliation is transverse K\"ahler and homologically orientable, then by \cite{mol,Masa} we have the volume form $\nu$ on the leaf such that $d\nu$ is basic. And then we have the solution to the transverse complex Monge-Amp\`ere equation by \cite{aziz}. So, the results of the Theorem \ref{mt} and the Remark \ref{rem:more-gen-comm-hessian} are true for higher dimensional foliations which are homologically orientable transverse K\"ahler as well. Just for simplicity and readability we assume that the dimension of the foliation is one. 
\end{rem}	
	Next, we discuss an application of the above estimate to the regularity of conical Ricci flat metrics on an affine $\QQ$-Gorenstein $\mathbb{T}$-variety. We first recall some definitions and introduce some notations. 
	\begin{defn}
	Let $Y$ be a normal affine variety, and let $\mathbb{T}_{\CC} \approx (\CC^*)^k$ be a complex torus. We say that a $\mathbb{T}_\CC$-action on $Y$ is \textit{good} if it is effective and there is a unique point $y_0\in Y$ that lies in the orbit closure of any $\mathbb{T}$-orbit. 
\end{defn}
\begin{defn}
	A normal variety $Y$ is called \textit{$\QQ$-Gorenstein} if the canonical bundle $K_Y$ is $\mathbb{Q}$-Cartier ie. there exist a integer $m$ such that $K_Y^{\otimes m}$ is Cartier. We say that it has Kawamata log terminal $(klt)$ singularities if there exists a resolution $\pi:\tilde Y\rightarrow Y$ is a resolution with $$K_{\tilde Y}  = \pi^*K_Y + \sum_{i}a_iE_i,$$ where $a_i > - 1$ for all $i$. 
\end{defn}

We say that $Y$ is a $\mathbb{T}$-variety, if it admits a good action of $\mathbb{T}_{\CC}$ and is $\QQ$-Gorenstein. In this note we will be concerned with  {\em polarized} $\TT$-varieties $(Y,\xi)$, namely $\TT$-varieties decorated with a {\em Reeb vector field}. Throughout, we will also assume that $Y$ has klt singularities. If we let $J$ denote the complex structure on the regular set $Y_{reg}$, then $-J\xi$ is a vector field that generates a holomorphic $\RR_{+}$ action with a repulsive fixed point $y_0$.
 
\begin{defn}
	\begin{enumerate}
		\item A function $f:Y\to \RR_+$ is said to be \textit{radial with respect to $\xi$} or $\xi$-radial if $$\sL_\xi f=0 \text{ and } \sL_{-J\xi}f=f.$$ 
		\item We say that $f$ is a \textit{conical K\"ahler potential with respect to $\xi$} if it can be expressed as $f=r^2$ for a $\xi$-radial function $r$ on $Y$ and $f$ is smooth and strictly plurisubharmonic on $Y_{reg}$ and continuous on $Y.$ For such a function $f$, we say that $$\omega = \ddbar f$$ is a conical K\"ahler metric on $Y$.
	\end{enumerate}
\end{defn}
	In this note, we are interested in {\em Ricci flat} conical K\"ahler metrics. We say that a $\mathbb{T}$-invariant conical K\"ahler metric $\omega=\ddbar r^2$ on $Y_{reg}$, for some $r\in C^0(Y)$ is  Ricci flat if there exist a $\mathbb{T}$-invariant no-where vanishing section $\Omega\in H^0(Y,mK_Y)$, such that \[\label{flat} \omega^{n+1}=\left((\sqrt{-1})^{m(n+1)^2}\Omega\wedge\bar{\Omega}\right)^\frac{1}{m}, \text{ for some } m\geq 1 \]  in the sense of Bedford-Taylor(cf. \cite{BT}). We say that $\omega$ is weakly Ricci flat if $r$ is only locally bounded. Note that the above equation still makes sense.

	Weak Ricci flat metrics arise naturally in the study of tangent cones of Gromov-Hausdorff limits of non-collapsed K\"ahler-Einstein manifolds on Fano manifolds (cf. \cite{DS}). If the link of the tangent cone is smooth, or equivalently if the only singularity of the tangent cone is at the vertex, then the link is a Sasakian manifold with a Sasakian-Einstein metric. Building on the pioneering work of  Martelli-Sparks-Yau \cite{MSY} on the volume functional, Collins and Szekelyhidi \cite{CS18} extended the notion of $K$-stability to affine Calabi-Yau cones, and proved  that a polarized affine $\TT$-variety $(Y,\xi)$ with an isolated singularity $y_0$ admits a Ricci flat K\"ahler cone metric if and only it is $K$-stable. The reader can refer to the excellent expository article \cite{LLX} for more details about this circle of ideas relating normalised volume to Ricci flat cones. For the case of toric manifolds, Berman \cite{B20} extended this correspondence to general toric $\mathbb{T}$-varieties with possibly non-isolated singularities. Berman first constructed a weak Ricci flat cone metric, and then proved that the metric was smooth on the regular locus.  
	
	It is natural to expect that any weak Ricci flat cone metric is in fact smooth on $Y_{reg}$ (cf. \cite{LLX}). Indeed as an immediate consequence of our main theorem we obtain the following corollary. 
\begin{cor}\label{main thm}
	Let $(Y,\mathbb{T},\xi)$ be a complex affine normal variety which is $\QQ$-Gorenstein. And $\omega=\ddbar r^2$ is a weak Ricci-flat K\"ahler metric which is conical with respect to $\xi$ on $Y$. Then $r^2$ is smooth on $Y_{reg}.$
\end{cor}
 	That the corollary follows from the main theorem is rather standard (cf. \cite{EGZ2}). While we obtained the $L^\infty$-estimate for \textit{transverse nef} classes in Theorem \ref{mt}, to prove the corollary, we only need an estimate for non-negative transverse class.  In the context of transverse K\"ahler manifold, as noted above, this appeared in \cite{B20} for the toric case, and recently in \cite{TTR} for the general case. But as noted above, the main difference with \cite{TTR} is that our proof is a completely PDE proof. We outline the general argument in Section \ref{potential} for the convenience of the reader.  
	
	This note is organized in the following way: Section \ref{pre} provides an introduction to essential concepts in transverse K\"ahler geometry and affine $\QQ$-Gorenstein varieties. Moving forward to Section \ref{infty estimate}, we prove our main theorem, denoted as Theorem \ref{mt}. And in Section \ref{potential}, we establish the smoothness of the conical Calabi-Yau potential on the regular locus of the $\QQ$-Gorenstein affine variety, as articulated in Corollary \ref{main thm}.

\section{Preliminaries}\label{pre}

\subsection{Preliminaries on Transverse K\"ahler Geometry:}
	In this section, we introduce some definitions and preliminary results used in the note. Our introduction is very minimal, and we refer the reader to the monograph \cite{BG08}	by Boyer and Galicki for more on Sasaki manifolds and transverse K\"ahler manifolds and the book by Molino \cite{mo} for more on foliations.
\begin{defn}\label{foliation}
\begin{enumerate}
	\item	A dimension-$k$ \textit{foliation of a manifold} $S$ is described by the following data:
\begin{itemize}
	\item an open cover $\{U_i\}$, $S=\underset{i}{\cup}U_i.$
	\item a smooth manifold $T$ of dimension $dim(S)-k$ and smooth submersions $\ga_i:U_i\to T$
	\item and if $U_i\cap U_j\neq \emptyset$, then there exist a local diffeomorphism $\pi_{ij}:\ga_i(U_i)\to \ga_j(U_j)$ such that $\pi_{ij}\circ\ga_i=\ga_j. $
\end{itemize}
\item If $T$ is a complex manifold and the local diffeomorphisms $\pi_{ij}$ are holomorphic, then we say that $S$ is a transverse holomorphic manifold.
\end{enumerate}
\end{defn}	
	And a foliation defines a partition of the manifold $S$ and it is denoted by $\sF.$
\begin{defn}
	A $p$-form $\al$ on a foliated manifold $(S,\sF)$ is said to be a \textit{basic form} if $$i_v\al=0 \text{ and }\sL_v\al=0, \text{ for all } v\in T\sF.$$
\end{defn}
	If the manifold $T$, mentioned in the Definition \ref{foliation} is a complex manifold, then there exist an endomorphism $\Phi$ on $TS$ such that $$ \sL_v\Phi=0, \text{ for all } v\in T\sF, \Phi|_{T\sF}\equiv 0\text{ and } \Phi^2=-id \text{ on } TS/\sF. $$
\begin{defn}\label{tr kahler}
	Let $(S,\sF)$ be a foliated manifold, we say $S$ is a \textit{transverse K\"ahler manifold} if $(S,\sF)$ is a transverse holomorphic manifold and there exist a basic $(1,1)$ form $\omega$ which is positive definite on $(TS/\sF)^{1,0}.$
	And we call $\omega$ a transverse K\"ahler metric.
\end{defn}
	From the definition we can see that the exterior differential operator $d$ maps basic form to a basic form, so we can define the cohomology of basic forms. Let $\Omega^k_B(S,\RR)$ be the set of all basic forms which is a subset of $\Omega^k(S,\RR)$ and we define $$H^k_B(S,\RR):=\frac{ker(d:\Omega^k_B(S,\RR)\to\Omega^{k+1}_B(S,\RR))}{im(d:\Omega^k_B(S,\RR)\to\Omega^{k+1}_B(S,\RR))}.$$
\begin{defn}
	Let $S$ be a compact manifold of dimension $m+k$ with a dimension-$k$ foliation $\sF$, then we say that this foliation $(S,\sF)$ is \textit{homologically orientable} if $H^{m}_B(S,\RR)\neq\emptyset.$
\end{defn}
	A homologically orientable transverse K\"ahler manifold has almost all the properties satisfied by the K\"ahler manifolds. For example, the $\d\dbar$-lemma, Lefschetz decompostion, Hodge theorem and so on. See \cite{aziz,BG08,Masa,noz}, for further discussion on this.
\begin{rem}
	Not all transverse K\"ahler manifolds are homologically orientable. A counter example was given by Carri\`ere in \cite{car}, see \cite[Appendix A]{mo} also.
\end{rem}
\begin{defn}
	Let $(S,\sF,\omega)$ be a compact transverse K\"ahler manifold with $\omega$ a basic transverse K\"ahler form. Then
\begin{enumerate}
	\item $[\al]\in H^{1,1}_B(S,\RR)$ is \textit{basic K\"ahler} if it contains a smooth representative which is transverse K\"ahler, that is, there exist a smooth basic function $\vp$ such that $\al+\ddbar \vp$ is a transverse K\"ahler metric.
	\item $[\al]\in H^{1,1}_B(S,\RR)$ is said to be \textit{basic nef} if for every $\vep>0$ the class $[\al]+\vep[\omega]$ contains a smooth transverse K\"ahler metric.
\end{enumerate}
\end{defn}
\begin{defn}
	Let $(S,g)$ be a $2n+1$-dimension Riemannian manifold. We say $S$ is a \textit{Sasaki manifold} if the Riemannian cone $(C(S),\bar{g})=(S\times\RR_+,dt^2+t^2g)$ is a K\"ahler manifold.
\end{defn}
	We can identify $S$ in $C(S)$ by $\{1\}\times S$. Let $J$ be the complex structure of $C(S)$ and define a vector field $\xi$ and the $1$-form $\eta$ by $$\xi=J\left(t\frac{\d}{\d t}\right),~\eta(\cdot)=\bar{g}(\xi,\cdot).$$
The restriction of $\xi$ and $\eta$ to $S$ satisfies the following properties:
\begin{itemize}
	\item $\xi$ is a Killing vector field
	\item the integral curve of $\xi$ is a geodesic
	\item $\eta(\xi)=1 $,$d\eta(\xi,v)=0$, for any vector field $v$ on $S$ and $d\eta$ is a transverse K\"ahler form.
\end{itemize}
	The vector field $\xi$ defines a foliation of the manifold $S$ called as the \textit{Reeb foliation} we denote it by $\sF_\xi$.
	Let $\theta$ be any basic $(1,1)$ form, then
\begin{gather*}
	\Nu_\theta(x):=\sup\{u(x):u\in PSH(S,\xi,\theta),u\leq 0\} \text{ and } \\ \hat{\Nu}_\theta(x):=\sup\{u(x):u\in PSH(S,\xi,\theta)\cap C^\infty(S),u\leq 0\}
\end{gather*}
	are the envelope of the potentials. And we define $$\Nu_\theta^f:=\sup\{u(x):u\in PSH(S,\xi,\theta),u\leq f\},$$ for any basic function $f.$ And by \cite[Proposition 3.17]{HL21}, $\Nu_\theta,\hat{\Nu}_\theta,\Nu_\theta^f\in PSH(S,\xi,\theta).$
	
	The following proposition and theorem are essentially transverse K\"ahler versions of \cite[Proposition 2.3]{B19} and \cite[Theorem 3.3]{B19} respectively, which provides the approximations of the envelopes that will be needed in the course of proof of our main theorem, and its corollary.
\begin{prop}
	\label{envelope approximation}Let $[\theta]$ and $[\omega]$ be any two transverse K\"ahler class on a transversely  K\"ahler and homologically orientable manifold $(\tilde{S},\tilde{\sF}_\xi,g,\xi)$. Then the smooth solution $u_\be$ of the equation \[\label{be cma} (\theta+\ddbar u_\be)^n\wedge\eta=e^{\be u_\be}\omega^n\wedge\eta \] satisfies $$ \|u_\be-\hat{\Nu}_\theta\|_{C^0(\tilde{S})}\leq \frac{A\log \be}{\be}, $$ where the constant $A$ depends on an upper bound on $\frac{\theta^n\wedge\eta}{\omega^n\wedge\eta}.$
\end{prop}
\begin{proof} 
	Let $x\in \tilde{S}$ be such that $\sup_{s\in \tilde{S}}u_\be(s)=u_\be(x).$ Then from \eqref{be cma} and the fact that $\ddbar u_\be(x) \leq0$, we get 
	$\theta^n\wedge\eta(x)\geq e^{\be u_\be(x)}\omega^n\wedge\eta(x)$, thus
	$$u_\be(s)\leq \frac{C}{\be}, \text{ where } C=\sup_{\tilde{S}}\log\left|\frac{\theta^n\wedge\eta}{\omega^n\wedge\eta}\right|<\infty. $$
	Now, since $u_\be-\frac{C}{\be}\in PSH(\tilde{S},\theta,\xi)$ by definition  $$ u_\be-\frac{C}{\be}\leq \hat{\Nu}_\theta. $$
	Now let $u\in PSH(\tilde{S},\theta,\xi)\cap C^\infty(\tilde{S})$ and $v$ be a fixed smooth basic and strictly subharmonic function on $\tilde{S}.$ Since $[\theta]$ is a transverse K\"ahler class such a choice is possible.
	And $u_{\vep,\delta}=(1-\vep)u+\vep\ v-\delta\in PSH(\tilde{S},\theta,\xi)\cap C^\infty(\tilde{S}),$ where $\vep=\be^{-1}$ and $\delta$ will be chosen later.
	Then 
\begin{align*}
	\left((1-\vep)\theta+\vep\theta+\ddbarb u_{\vep,\delta}\right)^n\wedge\eta=&\sum_{k=0}^{n}\binom{n}{k}(1-\vep)^{n-k}\vep^k\theta_u^{n-k}\theta_v^k\wedge\eta\\
	\geq&\vep^n\theta_v^n\wedge\eta\\
		\geq& e^{-\delta\be}\omega^n\wedge\eta, 
\end{align*}
	where we chose $ \delta=C'\be^{-1}\log\be$ for some large appropriate constant $C'.$
	Thus 
	$\theta_{u_{\vep,\delta}}^n\wedge\eta\geq e^{\be u_{\vep,\delta}}\omega^n\wedge\eta,$ since $u_{\vep,\delta}\leq 0.$
	We already have $\theta_{u_{\be}}^n\wedge\eta=e^{\be u_\be}\omega^n\wedge\eta$, thus by comparison principle we can get $u_{\vep,\delta}\leq u_\be $ as we see below.
\begin{cl}
	$u_{\vep,\delta}\leq u_\be.$
\end{cl}
	Suppose that $\sup_{\tilde{S}}(u_{\vep,\delta}-u_\be)=(u_{\vep,\delta}-u_\be)(x).$
	Then $\ddbarb(u_\be-u_{\vep,\delta})(x)\leq 0,$ that is $\theta_{u_{\vep,\delta}}(x)\leq \theta_{u_\be}(x)$. Thus $$\frac{\theta_{u_{\vep,\delta}}^n\wedge\eta(x)}{\theta_{u_{\be}}^n\wedge\eta(x)}\leq 1.$$ But $$ e^{\be(u_{\vep,\delta}-u_\be)(x)} \leq\frac{\theta_{u_{\vep,\delta}}^n\wedge\eta(x)}{\theta_{u_{\be}}^n\wedge\eta(x)} \leq  1 .$$ Thus $u_{\vep,\delta}-u_\be\leq 0$ on $\tilde{S}.$
	
	This implies that $\frac{\be-1}{\be}u\leq u_\be-\frac{1}{\be}v+\frac{C'\log \be}{\be}.$ Since $v\in C^\infty(\tilde{S})$, we get $u\leq \frac{\be}{\be-1}u_\be+\frac{C_1\log\be}{\be},$ for some constant $C_1.$
	
	And by definition of $\hat{\Nu}_\theta$ we have $$ u_\be-\frac{C}{\be}\leq \hat{\Nu}_\theta\leq \frac{\be}{\be-1}u_\be+\frac{C_1\log\be}{\be}. $$ Thus $$\|u_\be-\hat{\Nu}_\theta\|_{C^0(\tilde{S})}\leq \frac{A\log \be}{\be}, $$ follows from applying the inequality $u_\be\leq \frac{C}{\be}$ to the inequality  above.
\end{proof} 

\begin{thm}\label{approx}
	Let $\left(\tilde{S},\tilde{\sF}_\xi\right)$ be a transverse K\"ahler manifold which is homologically orientable and $u\in PSH(\tilde{S},\xi,\theta)$ then there exist a sequence $\{u_j\}_{\{j\in \mathbb{N}\}}$ such that $u_j\in PSH(\tilde{S},\xi,\theta)\cap C^\infty(\tilde{S})$ and $u_j\to u$ uniformly, for any basic $(1,1)$-form $\theta.$
\end{thm}
\begin{proof}
	Since $u$ is upper semi-continuous basic function, it can be written as a decreasing limit of basic smooth functions $\{f_j\}_{j\in \NN}.$ And we let $\vp_j:=\hat{\Nu}_\theta^{f_j}$, then by definition $f_j\geq\vp_j\geq u.$ Moreover, for any fixed $x\in \tilde{S}$ and an $\ep>0$ we can find $j_\ep\in\NN$ such that $$u(x)\leq \vp_j(x)\leq f_j(x)\leq u(x)+\ep, \text{ for all }j\geq j_\ep.$$ Thus $\vp_j$ decreases to $u$. 
	Now, to find a sequence of basic smooth psh-functions decreases to $u$ we consider the following transverse Complex Monge-Amp\`ere equations $$(\theta+\ddbarb\vp_{\be,j})^n=e^{\be(\vp_{\be,j}-f_j)}\omega^n $$ and $\vp_{\be,j}$ be the solutions which exists by the transverse Calabi-Yau theorem(cf.\cite{aziz}). And by Proposition \ref{envelope approximation}, we can see that $\vp_{j,\be}$ converges uniformly to $\vp_j$ as $\be\to \infty.$ Hence, for an appropriate choices of $\be_j\in\NN$ and $\ep_j>0$ such that $\be_j\to \infty$ and $\ep_j\to 0$ as $j\to \infty$ and if we let $$u_j:=\vp_{j,\be_j}+\ep_j$$ satisfies the necessary conditions we expected.
\end{proof}
\begin{rem}
	By the Theorem \ref{approx} and Proposition \ref{envelope approximation}, we can see that $\Nu_\theta^f=\hat{\Nu}_\theta^f$.
\end{rem}	
\subsection{Affine $\QQ$-Gorenstein variety:}
	The cone over a Sasaki manifold is an affine variety, we will see in this section that a link of an affine $\QQ$-Gorenstein variety can also given the structure of a Sasaki manifold. Let $Y$ be any affine $\QQ$-Gorenstein variety with non-isolated singularity and a good torus $\mathbb{T}$ action, and we let $r$ be a K\"ahler potenital on $Y.$ Then $$\eta=J(d\log r) \text{ and } d\eta=\ddbar\log r$$ defines a $\mathbb{T}$-invariant $1$-form and $2$-form respectively such that $\eta(\xi)=1$ and we define the link $S$ of $Y$ by $$S:=(Y\smallsetminus\{y_0\})/\RR_+,$$ where $y_0$ is the fixed point of $\TT$ and the action of $\RR_+$ is generated by $-J\xi$, and $J$ is the complex structure of $Y.$
	
	Then the restriction of $d\eta$ to the link $S$ is transverse K\"ahler form on $S.$ If we define the Riemannian metric $g$ on $S$ by $$g(\cdot,\cdot)=\eta\otimes\eta(\cdot,\cdot)+\frac{1}{2}d\eta(\cdot,\Phi\cdot),$$ then 
	$(S_{reg},g,\xi,\eta)$ is a Sasaki manifold, where $S_{reg}=S\cap Y_{reg}.$

	As we deal with the affine $\QQ$-Gorenstein varieties with the non-isolated singularities, to apply the PDE methods efficiently it is beneficial to work on smooth varieties. So, we take the resolution and try to convert the problem to the resolution as we will see in the upcoming section and in the subsection \ref{cone to link}.
\subsubsection{Resolution of $Y$:}	
	By Lemma 2.5 of \cite{CS18} we have an equivariant embedding of $(Y,\mathbb{T})$ in $\CC^N$ such that $Y$ does not contained entirely in a hyperplane. Denote by $\overline{Y}$ the Zariski closure of $Y$ in $\PP^N$ under the standard embedding of $\CC^N$ in $\PP^N.$ Then by \cite[Theorem 36]{kol}, there exist a $\mathbb{T}$-equivariant resolution $\pi:X\to \overline{Y}$ such that the inverse image of the singular locus of $\overline{Y}$ coincides with the support of an effective divisor $E$ on $X$. Then we let $$\tilde{Y}:=\pi^{-1}(Y)\subset X,$$ be the resolution of $Y.$
	
	So, the action of $\mathbb{T}$ on $Y$ can be extended to $\tilde{Y}$ as well. Since $\xi$ and $-J\xi$ are vector fields on $Y_{reg}$, it can be pulled back as vector fields to $\tilde{Y}$ which we again denote by  $\xi$ and $-J\xi$ respectively. And we let $$\mathcal{U}:=\pi^{-1}(Y_{reg}).$$
	
\begin{lem}[Lemma 4.12 of \cite{B19}]\label{metric on desing}
	There exist a $\mathbb{T}$-invariant K\"ahler form $\hat{\chi}$ on $\tilde{Y}$ and a smooth $\mathbb{T}$-invariant function $\rho$ on $\mathcal{U}$ such that $\rho\to -\infty$ at $\d\mathcal{U}$ and $$\ddbar\rho=\hat{\chi}$$ on $\mathcal{U}.$
\end{lem}
\begin{proof}	
	Since $\pi:X\to \overline{Y}$ is a $\mathbb{T}$-equivariant resolution and $\pi$ is relatively ample, by Kodaira's lemma there exists a positive number $\vep\in\QQ$ such that the $\QQ$-line bundle $$A:=\pi^*\mathcal{O}(1)-\vep[E]$$ over $X$ is ample. And let $\tilde{Y}:=\pi^{-1}(Y)$.
	
	Now fix a $\mathbb{T}$-invariant metric $h_A$ on $A$ whose curvature form defines a K\"ahler metric $\hat{\chi}$ on $X$. Since $\pi^*\mathcal{O}(1)=A+\vep[E]$ and $\mathcal{O}(1)$ is a trivial line bundle on the open set $Y\subset \PP^N$, we let $h_1$ a hermitian metric on $\pi^*\mathcal{O}(1)$ and $s_1$ a trivializing section of $\mathcal{O}(1)\to Y$. And $h_E$ be the hermitian metric on the line bundle $[E]$ such that $h_A=h_1\otimes (\vep h_E)^{-1}$ with $s_E$ be the section of $[E]$ which defines the divisor $E$ then we define $$ \rho=-\log\left(\|\pi^*(s_1)\|_{h_1}^2e^{-\vep\log\|s_E\|^2_{h_E}}\right) $$ which satisfies the required properties of the lemma.
\end{proof}
	We have a K\"ahler manifold $(\tilde{Y},\hat{\chi})$ such that $\hat{\chi}$ is $\mathbb{T}$-invariant. Also, we can think of $\tilde{Y}\smallsetminus\pi^{-1}(y_0)$ as a transverse K\"ahler manifold with the foliation generated by two commuting vector fields $\xi$ and $J\xi$ on $\tilde{Y}$. And we denote this foliation on $\tilde{Y}\smallsetminus\pi^{-1}(y_0)$ by $\sF.$
	
	Next, we state a theorem that provides us with a transverse K\"ahler form on $(\tilde{Y}\smallsetminus\pi^{-1}(y_0),\sF).$
	This result is an application of the previous lemma, proven in \cite[Proposition 4.3]{B20}. So, we skip the proof and an interested reader can refer \cite[Propostion 4.3]{B20} or \cite[Lemma 3.3]{TTR}.
\begin{thm}\label{tr metric on sing}
	There exists a basic transverse K\"ahler metric $\omega_b$ on $\tilde{Y}\smallsetminus\pi^{-1}(y_0)$ and an $\xi$-equivariant function $\Phi_b$ on $\tilde{Y}\smallsetminus\pi^{-1}(y_0)$(i.e., $\sL_\xi\Phi_b$=0 and $\sL_{-J\xi}\Phi_b=2$ ) such that $$\ddbar\Phi_b=\ddbarb\Phi_b=\omega_b$$
\end{thm}	
\begin{rem}
	The properties of $\Phi_b$ mentioned in the above theorem is crucial while derive the second order estimates(cf. Lemma \ref{higher order}), where $\Phi_b$ is used as a barrier function.
\end{rem}
	So, we have a resolution $\tilde{Y}$ of the variety $Y$ and a 
	basic transverse K\"ahler form $\omega_b$ on $\tilde{Y}\smallsetminus\pi^{-1}(y_0).$ Now we look at the structure of the link of $\tilde{Y}\smallsetminus\pi^{-1}(y_0)$. We denote by $$ \tilde{S}:=\left(\tilde{Y}\smallsetminus \pi^{-1}\{y_0\}\right)/\mathbb{R}_+ $$ where the action of $\mathbb{R}_+$ on $\tilde{Y}\smallsetminus \pi^{-1}\{y_0\}$ is by the flow of $-J\xi$.
		
	We conclude this section by pointing out that set  $\tilde{S}:=(\tilde{Y}\smallsetminus\pi^{-1}\{y_0\})/\RR_+$ is actually a transverse K\"ahler manifold.	
	Since, $\pi$ is a holomorphic and a $\mathbb{T}$-equivariant map, $d\pi(\xi_p)=\xi_{\pi(p)},$ for every $p\in \tilde{Y}$, thus $\tilde{S}=\pi^{-1}(S)$. And hence, $\tilde{S}$ is a transverse holomorphic manifold with the foliation generated by the vector field $\xi.$ And by \cite[Lemma 4.3]{B20}, the form defined on $\tilde{S}$ by $$\chi=\omega_b|_{\tilde{S}},$$ gives us a transverse K\"ahler form on $\tilde{S}.$

	Moerover, $(\tilde{S},\tilde{\sF}_\xi)$ is homologically orientable as we will see later(see, Theorem \ref{h orientable}).
	Also, $\pi$ induces a $CR$-map which preserves the transverse holomorphic structure, which we again denote by $\pi$, $$\pi:\tilde{S}\to S$$ with $\pi|_{\tilde{S}_{reg}}:\tilde{S}_{reg}\to S_{reg}$ is a transverse bi-holomorphic map, where $\tilde{S}_{reg}:=\tilde{S}\cap\mathcal{U}$. Thus, the pull-back of basic $(p,q)$-forms in $S_{reg}$ using $\pi$ will be basic $(p,q)$-forms in $\tilde{S}.$
\section{$L^\infty$ estimate}\label{infty estimate}
	In this section, we establish the proof of our main theorem, Theorem \ref{mt}. As in \cite{GPT21,GPTW21}, we also use the particular auxiliary transverse Monge-Ampere equation to prove the necessary estimate. And the method of proof largely aligns with \cite{GPT21,GPTW21}, with necessary modifications to the transverse K\"ahler manifold setting. Before we proceed to prove it, we need some lemmas. The next lemma is basically due to De Giorgi, and we state it without proof. However, an interested reader can refer to \cite[Lemma 2]{GPT21} for the proof.
	
\begin{lem}\label{degiorgi lemma}
	Let $f:\RR_+\to\RR_+$ be a decreasing right continuous function with $\lim_{s\to\infty}f(s)=0.$ Assume that $rf(s+r)\leq B_0f(s)^{1+\delta_0}$ for some constant $B_0>0$ and all $s>0$ and $r\in [0,1]$. Then there exists a positive constant $s_\infty:=s_\infty(\delta_0,B_0,f)$ such that $f(s)=0,$ for all $s\geq s_\infty.$
\end{lem}
	Now let $u_{t,\be}$ be the solution of the transverse complex Monge-Amp\`ere equation $$(\omega_t+\ddbarb u_{t,\be})^n\wedge\eta=e^{\be u_{t,\be}}\chi^n\wedge\eta,$$ where the existence of $u_{t,\be}$ is given by transverse Calabi-Yau theorem(cf. \cite{aziz}). By Proposition \ref{envelope approximation} and Theorem \ref{approx}, $u_{t,\be}\to \Nu_t$ as $\be\to \infty$, where $\Nu_t:=\Nu_{\omega_t}$, this fact will be needed in the proof of the following lemma.
\begin{lem}
	\label{lemma nef} There are positive constants $ C=C(n,\omega,\chi) $ and $ \be:=\be(n,\omega,\chi,\xi) $ such that for any $ s>0 $ \[ \int_{\Omega_s}exp\left\{\be\left(\frac{-\vp_t+\Nu_t-s}{A_s^{1/(n+1)}}\right)^{\frac{n+1}{n}}\right\}\chi^n\wedge \eta \leq C exp(CE_t), \]
	where $ A_s=\int_{\Omega_s}(-\vp_t+\Nu_t-s)e^F\chi^n\wedge\eta $ is the energy of $ (\vp_t-\Nu_t+s) $ on the sub-level set $ \Omega_s:=\{\vp_t-\Nu_t\le -s\} $ and $E_t=\int_{\tilde{S}}(-\vp_t+\Nu_t)e^F\chi^n\wedge\eta.$
\end{lem}
\begin{proof}
	The key idea of the proof is to dominate the function (upto scaling by some constants) $(-\vp_t+\Nu_t-s)^{\frac{n+1}{n}}$ by some other basic pluri subharmonic function which is a solution to a variant of transverse complex Monge-Amp\`ere equation and apply the alpha invariant to that solution to get the desired inequality.
	For that we choose a sequence of smooth functions $\tau_k:\RR\to \RR_+$ defined by $$\tau_k(x)=\frac{1}{2}\left(\sqrt{x^2+k^{-1}}+x\right),$$ which  converges pointwise to $x\cdot\chi_{\RR_+}(x)$, where $\chi_{\RR+}$ is the characteristic function on $\RR_+$ as $k\to \infty$ and we solve an auxiliary transverse Monge-Amp\`ere equation 
	$$(\omega_t+\ddbarb \psi_{t,k})^n\wedge\eta=c_t\frac{\tau_k(-\vp_t+\Nu_t-s)}{A_{s,k,\be}}e^F\chi^n\wedge\eta, \quad \sup_{\tilde{S}}\psi_{t,k}=0, $$
	where $A_{s,k,\be}=\int_{\tilde{S}} \tau_k(-\vp_t+u_{t,\be}-s)e^F\chi^n\wedge\eta.$
	By definition $\psi_{t,k}\leq \Nu_t$ and by Proposition \ref{envelope approximation} and Theorem \ref{approx}, $u_{t,\be}$ converges to $\Nu_t$ uniformly as $\be \to \infty$, then we can take $\be$ large enough so that $$\psi_{t,k}\leq 		u_{t,\be}+1. $$
	Now we define a basic smooth function $$ \Phi=-\vep(-\psi_{t,k}+u_{t,\be}+1+\La)^{\frac{n}{n+1}}-(\vp_t-u_{t,\be}+s), $$ where the positive constants $\vep,\La$ will be choosen later. And let $$\Phi(x_0)=\sup_{x\in \tilde{S}}\Phi(x).$$ Now either $x_0\in \Omega_s^\circ$ or $x_0\in \tilde{S}\smallsetminus\Omega_s^\circ.$ Let us look at each case separately.\\
	\textbf{Case.1} Suppose that $x_0\in \Omega_s^\circ.$ Then let $\Delta_t(f)=n\frac{\ddbarb f\wedge\omega_{t,\vp_t}^{n-1}\wedge\eta}{\omega_{t,\vp_t}^n\wedge\eta}=\La_{\omega_{t,\vp_t}}\ddbarb f.$ Since $x_0$ is a point of maximum of $\Phi$, we 	get
\begin{align*}
	0\geq&\Delta_t\Phi(x_0)\\
		=&-\Delta_t\vp_t+\Delta_t u_{t,\be}-\frac{\vep n}{n+1}(-\psi_{t,k}+u_{t,\be}+1+\La)^{\frac{-1}{n+1}}(-\Delta_t\psi_{t,k}+\Delta_t u_{t,\be})\\
		&+\frac{\vep n}{(n+1)^2}(-\psi_{t,k}+u_{t,\be}+1+\La)^{-\frac{n+2}{n+1}}\La_{\omega_{t,\vp_t}}\left(\sqrt{-1}\d_b(\psi_{t,k}-u_{t,\be})\wedge\dbar_b(\psi_{t,k}-u_{t,\be})\right)\\
		\geq&-n+\La_{\omega_{t,\vp_t}}\omega_{t,u_{t,\be}}+\frac{\vep n}{n+1}(-\psi_{t,k}+u_{t,\be}+1+\La)^{\frac{-1}{n+1}}\La_{\omega_{t,\vp_t}}(\omega_{t,\psi_{t,k}}-\omega_{t,u_{t,\be}})\\
		=&-n+\frac{\vep n}{n+1}(-\psi_{t,k}+u_{t,\be}+1+\La)^{\frac{-1}{n+1}}\La_{\omega_{t,\vp_t}}\omega_{t,\psi_{t,k}}\\
		&+\left(1-\frac{\vep n}{n+1}(-\psi_{t,k}+u_{t,\be}+1+\La)^{\frac{-1}{n+1}}\right)\La_{\omega_{t,\vp_t}}\omega_{t,u_{t,\be}}\\
		\geq&-n+\frac{\vep n^2}{n+1}(-\psi_{t,k}+u_{t,\be}+1+\La)^{\frac{-1}{n+1}}\left(\frac{\omega_{t,\psi_{t,k}}^n\wedge\eta}{\omega_{t,\vp_t}^n\wedge\eta}\right)^{\frac{1}{n}}+\left(1-\frac{n\vep}{n+1}\La^{\frac{-1}{n+1}}\right)\La_{\omega_{t,\vp_t}}\omega_{t,u_{t,\be}}\\
		\geq&-n+\frac{\vep n^2}{n+1}(-\psi_{t,k}+u_{t,\be}+1+\La)^{\frac{-1}{n+1}}\left(\tau_k(-\vp_t+u_{t,\be}-s)A_{s,k,\be}^{-1}\right)^{\frac{1}{n}}\\
		&+\left(1-\frac{n\vep}{n+1}\La^{\frac{-1}{n+1}}\right)\La_{\omega_{t,\vp_t}}\omega_{t,u_{t,\be}}\\
		\geq&-n+\frac{\vep n^2}{n+1}(-\psi_{t,k}+u_{t,\be}+1+\La)^{\frac{-1}{n+1}}\left(-\vp_t+u_{t,\be}-s\right)^{\frac{1}{n}}(A_{s,k,\be})^{\frac{-1}{n}}\\
		&+\left(1-\frac{n\vep}{n+1}\La^{\frac{-1}{n+1}}\right)\La_{\omega_{t,\vp_t}}\omega_{t,u_{t,\be}}
\end{align*}
	Now we choose the constant $\La $ so that the last term in the sum is zero, that is $$\La=\left(\frac{n\vep}{n+1}\right)^{n+1}, $$ then at the point $x_0\in \Omega_s$, we get that $$ -(\vp_t-u_{t,\be}+s)\leq \left(\frac{n+1}{n\vep}\right)^n A_{s,k,\be}(-\psi_{t,k}+u_{t,\be}+\La+1)^{\frac{n}{n+1}} $$
	now if we choose $$ \vep^{n+1}=\left(\frac{n+1}{n}\right)^n A_{s,k,\be},	$$ then the above inequality becomes $$ -(\vp_t-u_{t,\be}+s)\leq\vep(-\psi_{t,k}+u_{t,\be}+\La+1)^{\frac{n}{n+1}}. $$ That is $\Phi(x_0)\leq 0.$\\
	\textbf{Case.2} Suppose $x_0\in \Omega_s^\circ.$ Then by definition of $\Omega_s$, we get $$\Phi(x_0)\leq -(\vp_t-u_{t,\be}+s)\leq -U_t+u_{t,\be}\leq \ep_\be, $$
	where we chose $\ep_\be>0$  using Proposition \ref{envelope approximation} so that $\ep_\be\to 0$ as $\be\to \infty.$
	From both the cases, we can get that on $\tilde{S}$ we have the following inequality $$ (\vp_t+u_{t,\be}-s)^{\frac{n+1}{n}}\leq C_n(A_{s,k,\be})^{\frac{1}{n}}(-\psi_{t,k}+u_{t,\be}+1+A_{s,k,\be})+\ep_\be^{\frac{n+1}{n}}. $$
	Since $u_{t,\be}\to U_t$ uniformly as $\be\to \infty$, by letting $\be\to\infty$ in the above inequality we get $$ (-\vp_t+\Nu_t-s)^{\frac{n+1}{n}}\leq C_n(A_{s,k})^{\frac{1}{n}}(-\psi_{t,k}+\Nu_t+1+A_{s,k}) , $$ where $A_{s,k}=\int_{\tilde{S}}\tau_k(-\vp_t+\Nu_t-s)e^F\chi^n\wedge\eta .$
	And since $\Nu_t\leq 0$, we can reduce the above inequality becomes \[\label{psh inequality}\frac{(-\vp_t+\Nu_t-s)^{\frac{n+1}{n}}}{(A_{s,k})^{\frac{1}{n}}}\leq C_n(-\psi_{t,k}+1+A_{s,k}). \]
	By \cite[Proposition 3.3]{alpha invariant}, there exists positive constants $\al$ and $C'$ such that $$\int_{\tilde{S}}e^{-\al\psi_{t,k}}\chi^n\wedge\eta<C'.$$ And let $\be=\frac{\al}{C_n}$, then there exist a positive constant  $C$ such that $$ \int_{\Omega_s}\exp{\left(\be\frac{(-\vp_t+U_t-s)^{\frac{n+1}{n}}}{(A_{s,k})^{\frac{1}{n}}}\right)}\chi^n\wedge\eta\leq C\exp(CA_{s,k}). $$
	Since $A_{s,k}\to A_s$ as $k\to\infty$, letting $k\to\infty$ in the above inequality implies that $$ \int_{\Omega_s}\exp{\left(\be\frac{(-\vp_t+U_t-s)^{\frac{n+1}{n}}}{(A_s)^{\frac{1}{n}}}\right)}\chi^n\wedge\eta\leq C\exp(CA_s). $$ Then the desired inequality follows from the inequality $A_s\leq E_t$ for any $s>0.$
\end{proof}
	Next we prove that the constants $E_t$'s are bounded uniformly for $t\in (0,1].$
\begin{lem}
	\label{et bound}There is a constant $C:=C(n,\omega,\chi,\xi,\|e^F\|_{L^1(\log L)^p})$ such that $$E_t=\int_{\tilde{S}}(-\vp_t+U_t)e^F\chi^n\wedge\eta\leq C, \text{ for all } t\in(0,1].$$
\end{lem}
\begin{proof}
	We have $\omega_t=\omega+t\chi$, then we can find a constant $a>0$ such that $\omega_t< a\chi$ for all $t\in(0,1]$ then, $\vp_t\in PSH(S,\omega_t,\xi)$ implies $\vp_t\in PSH(S,a\chi,\xi)$. Then by the basic $\al$-invariant, \cite[Proposition 3.3]{alpha invariant} there exist an $\al:=\al(S,\xi,\chi)>0$ such that $$\frac{1}{V_t}\int_{\tilde{S}}\exp\left(-\log\frac{\omega_{t,\vp_t}^n\wedge\eta}{a^n\chi^n\wedge\eta}-\al\vp_t\right)\omega_{t,\vp_t}^n\wedge\eta=\frac{1}{V_t}\int_{\tilde{S}}a^ne^{-\al\vp_t}\chi^n\wedge\eta\leq \frac{a^nC(n,\xi,\chi)}{V_t}.$$ 
	By Jensen's inequality we get $$\frac{1}{V_t}\int_{\tilde{S}}\left(-\log\frac{\omega_{t,\vp_t}^n\wedge\eta}{a^n\chi^n\wedge\eta}-\al\vp_t\right)\omega_{t,\vp_t}^n\wedge\eta\leq\log C(n,\xi,\chi)-\log V_t,$$
	this implies
\begin{align*}
	\frac{1}{V_t}\int_{\tilde{S}}(-\al\vp_t)\omega_{t,\vp_t}^n\wedge\eta\leq&\frac{1}{V_t}\int_{\tilde{S}}\log(c_ta^{-n}e^F)\omega_{t,\vp_t}^n\wedge\eta+\log C(n,\xi,\chi)-\log V_t\\
	\leq &\log(c_ta^{-n})+\frac{c_t}{V_t}\int_{\tilde{S}} Fe^F\chi^n\wedge\eta+\log C(n,\xi,\chi)-\log V_t\\
	\leq &\log\frac{V_t}{a^n\int_{\tilde{S}}\chi^n\wedge\eta}+\frac{1}{\int_{\tilde{S}}\chi^n\wedge\eta}\|e^F\|_{L^1(\log L)^1}+\log C(n,\xi,\chi)-\log V_t\\
	\frac{1}{V_t}\int_{\tilde{S}}(-\al\vp_t)c_te^F\chi^n\wedge\eta\leq&\log\frac{C(n,\xi,\chi)}{a^n\int_{\tilde{S}}\chi^n\wedge\eta}+\frac{1}{\int_{\tilde{S}}\chi^n\wedge\eta}\|e^F\|_{L^1(\log L)^1}
\end{align*}
	thus we get
	$$\int_{\tilde{S}}(-\vp_t)e^F\chi^n\wedge\eta\leq \frac{\int_{\tilde{S}}\chi^n\wedge\eta}{\al}\log\frac{C(n,\xi,\chi)}{a^n\int_{\tilde{S}}\chi^n\wedge\eta}+\al^{-1}\|e^F\|_{L^1(\log L)^1}.$$
	Since $\vp_t\leq U_t\leq 0$ the last inequality yields
	$$\int_{\tilde{S}}(-\vp_t+U_t)e^F\chi^n\wedge\eta\leq C,$$ where the constant $C:=C(n,\xi,\chi,\|e^F\|_{L^1(\log L)^1}).$
\end{proof}
	Note that the constant $C$ in the above lemma does not depend on $ t $, which is a key point in establishing the $L^\infty$ estimate. Now we proceed to prove our main theorem.
	
\begin{proof}[\textbf{Proof of Theorem \ref{mt}}]
 	Define $h:\RR_+\to\RR_+$ by $h(x)=(\log(1+x))^p$. Then $h$ is a strictly increasing function with $h(0)=0$ and let $$v:=\frac{\be}{2}\left(\frac{-\vp_t+\Nu_t-s}{(A_s)^\frac{1}{n+1}}\right)^{\frac{n+1}{n}}$$ then by the generalized Young's inequality with respect to $h,$ for any $x\in \Omega_s,$ we have 
\begin{align*}
 	v(x)^pe^{F(x)}\leq&\int_{0}^{e^{F(x)}}h(y)dy+\int_{0}^{v(x)^p}h^{-1}(y)dy\\
 	\leq&e^{F(x)}(1+|F(x)|)^p+C(p)e^{2v(x)}.
\end{align*} This implies that
\begin{align*}
 	\int_{\Omega_s}v(x)^pe^{F(x)}\chi^n\wedge\eta\leq&\int_{\Omega_s}e^{F(x)}(1+|F(x)|)^p\chi^n\wedge\eta+\int_{\Omega_s}e^{2v(x)}\chi^n\wedge\eta\\
 	\leq& \|e^F\|_{L^1(\log L)^p}+C+Ce^{CE_t},
\end{align*}
	where the second inequality followed from Lemma \ref{lemma nef} and the constant $C=C(n,\xi,\chi,\omega)$. And by the definition of $v$ we get from the last inequality that $$\int_{\Omega_s}(-\vp_t+\Nu_t-s)^{\frac{(n+1)p}{n}}e^F\chi^n\wedge\eta\leq2^p\be^{-p}(A_s)^{\frac{p}{n}}(\|e^F\|_{L^1(\log L)^p}+C+Ce^{CE_t}).$$
	Now, 
\begin{align*}
 	A_s=&\int_{\Omega_s}(-\vp_t+\Nu_t-s)e^F\chi^n\wedge\eta\\
 	\leq &\left(\int_{\Omega_s}(-\vp_t+\Nu_t-s)^{\frac{(n+1)p}{n}}e^F\chi^n\wedge\eta\right)^{\frac{n}{(n+1)p}}\left(\int_{\Omega_s}e^F\chi^n\wedge\eta\right)^\frac{1}{q}\\
 	\leq&(A_s)^\frac{1}{n+1}\left(2^p\be^{-p}(\|e^F\|_{L^1(\log L)^p}+C+Ce^{CE_t}) \right)^{\frac{n}{(n+1)p}}\left(\int_{\Omega_s}e^F\chi^n\wedge\eta\right)^\frac{1}{q},
\end{align*}
	where $q>1$ satisfies $\frac{n}{(n+1)p}+\frac{1}{q}=1$ and for the second inequality we use the H\"older inequality. Thus the above inequality implies \[\label{fin ineq}A_s\leq \left(2^p\be^{-p}(\|e^F\|_{L^1(\log L)^p}+C+Ce^{CE_t})\right)^\frac{1}{p}\left(\int_{\Omega_s}e^F\chi^n\wedge\eta\right)^\frac{n+1}{nq}.\]
	By the definition of $q$, we get that $$\frac{n+1}{nq}=1+\frac{1}{n}-\frac{1}{p}=1+\delta_0,$$ where $\delta_0=\frac{p-n}{pn}>0.$ By Lemma \ref{et bound}, we can find a $B_0\in (0,\infty)$ such that $$\left(2^p\be^{-p}(\|e^F\|_{L^1(\log L)^p}+C+Ce^{CE_t})\right)^\frac{1}{p}<B_0, \text{ for all } t\in (0,1].$$ Then the inequality \eqref{fin ineq} can be rewritten as \[\label{degiorgi ineq}A_s\leq B_0\left(\int_{\Omega_s}e^F\chi^n\wedge\eta\right)^{1+\delta_0}.\] In $\Omega_{s+r}$, we have $-\vp_t+\Nu_t-s\geq r$ and  if we define the function $f:\RR_+\to\RR_+$ by $f(s)=\int_{\Omega_s}e^F\chi^n\wedge\eta$ then $f$ is decreasing, and using the inequality \eqref{degiorgi ineq} we get that $$rf(s+r)\leq B_0f(s)^{1+\delta_0},\ \text{ for any } r\in [0,1] \text{ and } s>0,$$ then by Lemma \ref{degiorgi lemma} we can find an $s_\infty>0$ such that $$f(s)=0, \text{ for all } s\geq s_\infty.$$ That is $$\int_{\Omega_s}(-\vp_t+\Nu_t-s)e^F\chi^n\wedge\eta=0, \text{ for all } s\geq s_\infty,$$ so $\Omega_s=\emptyset$ for all $s\geq s_\infty.$ Since $B_0$ is independent of $t$, so does the constant $s_\infty$. Thus $-\vp_t+\Nu_t\leq s_\infty$ and by definition of $\Nu_t$, we have $-\vp_t+\Nu_t\geq 0.$
\end{proof}
	We end this section with a remark about the sharp $L^\infty$ estimate for a general transverse complex Hessian equations.
\begin{rem}\label{rem:more-gen-comm-hessian} 
	General complex Hessian equations on Sasakian manifolds was studied by Feng and Zeng in \cite{FZ19}. We define $$\Nu_{t,k}:=\sup\{v :v\in SH_k(\tilde{S},\xi,\chi,\omega_t)\cap C^2(\tilde{S}),v\leq 0\}$$ where $v\in SH_k(\tilde{S},\xi,\chi,\omega_t)\cap C^2(\tilde{S})$ means that the eigenvalue vector $\la(\chi^{-1}(\omega_t+\ddbarb v))$ of the metric $\chi^{-1}(\omega_t+\ddbarb v)$	lies in $$\Ga_k:=\{x\in \RR^n:\sigma_i(x)>0,1\leq i\leq k\},$$ where $\sigma_i$ is the $i^{th}$ symmetric polynomial in $\RR^n$ and $[\chi]$ and $[\omega_t]$ are the transverse K\"ahler classes as it is defined in the introduction.	
	Then the general complex Hessian equations are defined as $(1\leq k\leq n)$ \[\label{hessian eqns}\begin{cases}
	(\omega_t\ddbarb\vp_t)^k\wedge\chi^{n-k}\wedge\eta=c_te^F\chi^n\wedge\eta,\\ \sup_{\tilde{S}}\vp_t=0 \text{ and } \la(\omega_{t,\vp_t})\in \Ga_k.
	\end{cases}\]
	Using the proof of the existence of the solution of the equation $\eqref{hessian eqns}$ on Sasaki manifolds in \cite[Corollary 1.5]{FZ19}, we can get the solution $\vp_t$ of the equation \eqref{hessian eqns} on the homologically orientable transverse K\"ahler manifold also. 
	We let $$E_t(\vp_t):=\int_{\tilde{S}}(-\vp_t+\Nu_{t,k})e^{\frac{n}{k}F}\chi^n\wedge\eta,$$ be the entropy functional corresponding the equation \eqref{hessian eqns}, then similar to \cite[Lemma 7]{GPT21} we can also have the inequality \[\label{et ineq} E_t(\vp_t)\leq \frac{c_t^n}{V_t}\|e^{\frac{n}{k}F}\|_{L^q}\|-\vp_t+\Nu_{t,k}\|_{L^p}, \] where $\frac{1}{p}+\frac{1}{q}=1.$
	And one of the crucial fact used in proving Theorem \ref{mt}, is that the approximation of the envelope (Proposition \ref{envelope approximation}, Theorem \ref{approx}). By adapting the proof of \cite[Lemma 2]{GPTW21} to the transverse K\"ahler setting, we can also prove as we proved Proposition \ref{envelope approximation} and Theorem \ref{approx} using the  maximum principle, that the solutions $u_\be$ of the equation $$(\omega_t+\ddbarb u_\be)^k\wedge\chi^{n-k}\wedge\eta=c_te^{\be u_\be}\chi^n\wedge\eta,$$ (which exists by \cite[Corollary 1.5]{FZ19}) converges uniformly to $\Nu_{t,k}$ as $\be\to \infty.$

	Now by modifying the proof of \cite[Lemma 8]{GPT21} to the transverse K\"ahler setting, we can prove that, for any $p\in (0,\frac{n}{n-k})$, there exists a constant $C:=C(n,p,\xi,\chi)$ such that $$\|-\vp_t+\Nu_{t,k}\|_{L^p(\chi^n\wedge\eta)}\leq C,$$ this with inequality \eqref{et ineq} implies that $E_t(\vp_t)$ bounded once the right hand-side $e^F\in L^q(\chi^n\wedge\eta)$.
	
	And then following the similar arguments as in Theorem \ref{mt}, modified to the equation \eqref{hessian eqns}, we have the sharp $L^\infty$ estimate for the transverse complex Hessian equations as a following theorem.
\begin{thm}
	For a given $p>n$, there exists a constant
	 $$ C:=C(n,p,\xi,\|e^{\frac{n}{k}F}\|_{L^1(\log L)^p},\frac{c_t}{\int_{\tilde{S}}\omega_t^k\wedge\chi^{n-k}\wedge\eta})$$ such that $$0\leq -\vp_t+\Nu_{t,k}\leq C.$$
\end{thm}
\end{rem}

\section{Regularity of the conical Calabi-Yau potential}\label{potential}

\subsection{From cone$(Y)$ to its link$(S)$:}\label{cone to link}
	In this section, we explain how we can convert the problem of finding the apriori estimates for the metric in the affine normal $\QQ$-Gorenstein variety to its link a transverse K\"ahler manifold. We accomplish this by exploiting the inherent symmetry of the cone.
	
	As an application of Berman's construction of a $\mathbb{T}$-invariant K\"ahler metric $\hat{\chi}$ on the resolution of $Y$(see, Lemma \ref{metric on desing}) and the basic transverse K\"ahler metric $\chi$ (see, Theorem \ref{tr metric on sing}), we have the following theorem and this will needed while obtaining the $L^\infty$-estimate.
\begin{thm}\label{h orientable}
	The foliated manifold $\left(\tilde{S},\tilde{\sF}_\xi\right)$ is a transverse K\"ahler manifold which is homologically orientable.
\end{thm}
\begin{proof}
	From the construction we can see that $\tilde{S}$ is transversely holomorphic manifold. Moreover, $\tilde{S}$ can be identified with the set $\{y\in \tilde{Y}\smallsetminus\pi^{-1}\{y_0\}:(f_\xi\circ\pi)(y)=1\}$ and the restriction of the metric $h(\cdot,\cdot)=\hat{\chi}(\cdot,J\cdot)$ constructed in Lemma \ref{metric on desing} to $\tilde{S}$ gives us a Riemannian metric $h$ which is Killing with respect to the vector field $\xi$.
	
	And we define a smooth nowhere vanishing $1$-form $\hat{\eta}$ on $\tilde{S}$ from $h$ by  $$\hat{\eta}(\cdot):=\frac{h(\xi,\cdot)}{h(\xi,\xi)}.$$
	Since $\pi:\tilde{S}\to S$ is a $\mathbb{T}$-equivariant map and the flow of $\xi$ on $S$ does not have any fixed point, $\xi$ is nowhere vanishing on $\tilde{S}$, so $\hat{\eta}$ is well defined and $d\hat{\eta}$ is a basic form.

	And by \cite[Lemma 4.3]{B20}, the form defined by $\chi:=\omega_b|_{\tilde{S}}$ is a basic transverse K\"ahler form, where $\omega_b$ is the basic transverse K\"ahler form defined in Theorem \ref{tr metric on sing}.
	
	Now we prove that $\tilde{S}$ is homologically orientable by proving that the form $\chi^n$ is basic top form which is not exact, that is $[\chi^n]\neq 0.$
	
	Suppose that $[\chi^n]=0,$ that is $\chi^n=d\ga,$ where $\ga$ is a basic $(2n-1)$-form. Then
	$$ 0=\int_{\tilde{S}}\ga\wedge d\hat{\eta}
		=\int_{\tilde{S}}\chi^n\wedge\hat{\eta}
		>0 $$
	which is a contradiction. The first equality is due to the fact that there is non-zero $(2n+1)$-basic form and the last inequality is because $$ \chi^n\wedge\hat{\eta}=det(\chi_{i\bar{j}})(\sqrt{-1})^ndz_1\wedge d\bar{z}_1\wedge dz_2\wedge d\bar{z}_2\wedge\cdots\wedge dz_n\wedge d\bar{z}_n\wedge dx $$ for the foliated coordinate chart $U$ with the coordinates $(x,z_1,z_2,\cdots,z_n)$ and $det(\chi_{i\bar{j}})$ is a positive function on  $U\subset \tilde{S},$ for any folilated coordinate chart.
	Thus $0\neq [\chi^n]\in H^{n,n}_B(\tilde{S},\mathbb{R})$, that is $(\tilde{S},\tilde{\sF}_\xi)$ is homologically orientable.
\end{proof}

\begin{rem}\label{rem sasaki}
	The transverse K\"ahler manifold $\tilde{S}$ may not be a Sasaki manifold, as the metric $\hat{\chi}=\ddbar\rho$ constructed in Lemma \ref{metric on desing} is not a cone metric $\left(\tilde{S}\times\mathbb{R}_+,dt^2+t^2{h|}_{\tilde{S}}\right)$ for the cone over $\left(\tilde{S}, {h|}_{\tilde{S}}\right)$. But, by Theorem \ref{h orientable} and \cite{Masa}, there exists a Riemannian metric on $\tilde{S}$ such that the flow of $\xi$ will be a geodesic and that metric may not be compatible with the transverse complex structure $\Phi$.
\end{rem}
	By Lemma 3.3 of \cite{B20}, there always exist a $\mathbb{T}$-invariant conical K\"ahler potential on $Y$ and we let $f_\xi$ be one such fixed $\mathbb{T}$-invariant conical K\"ahler potential on $Y$ with the corresponding transverse K\"ahler metric $\omega_T=\frac{1}{2}d(J(d\log f_\xi)|_{S_{reg}})=(\ddbarb\log f_\xi)|_{S_{reg}}$ on $S_{reg}$ and now, we look at the weak K\"ahler Ricci flat metric $\omega$ with the conical K\"ahler potential $r^2$ on $Y$ for which the equation \eqref{flat} is satisfied and $r\frac{\d}{\d r}=-J\xi.$	Then the function $r$ should be of the form
	$$r^2=f_\xi e^\vp,$$ where $\vp$ is some function defined on $Y$ which is invariant under the action of both $\xi$ and $-J\xi$ and locally bounded. Therefore, the smoothness of $\vp$ results in the smoothness of the conical potential $r^2.$ Thus, we try to prove the smoothness of $\vp.$

	Using the  $\mathbb{T}$-equivariant resolution of $Y$ which has klt-singularities, we have that \[\label{vol resolution} \pi^*K_Y=K_{\tilde{Y}}+\tilde{D}, \] where $\tilde{D}=D_+-D_-$ is the corresponding discrepancy divisor with $$D_+=\sum_{i}a_iD^+_i \text{ and }D_-=\sum_{j}b_jD^-_j$$ since $Y$ is $\QQ$-Gorenstein and has klt-singularities we have that $a_i>-1$ and $0<b_j<1$ and the sum is finite. The divisors $D^+_i$ and $D^-_j$ are $\mathbb{T}$-invariant and there exists sections $s^+_i$ and $s^-_j$ on the $\QQ$-line bundles which defines the divisors respectively.  Then by \eqref{vol resolution}, we have  $$(\sqrt{-1})^{(n+1)^2}\tilde{\Omega}\wedge\overline{\tilde{\Omega}}=\left(\prod_i\|s^+_i\|^{2a_i}_{h_i^+}\right)\left(\prod_j\|s_j^-\|^{-2b_j}_{h_j^-}\right)dV_{\tilde{Y}},$$
	where $dV_{\tilde{Y}}$ is a smooth $\mathbb{T}$-invariant positive volume form on $\tilde{Y}$ and $\tilde{\Omega}=\pi^*(\Omega)$.
	We define $$ \Psi^+=2\sum_{i}a_i\log\|s_i^+\|_{h_i^+} \text{ and } \Psi^-=2\sum_{j}b_j\log\|s_j^-\|_{h_j^-}. $$  then the above equation becomes $$(\sqrt{-1})^{(n+1)^2}\tilde{\Omega}\wedge\overline{\tilde{\Omega}}=e^{\Psi^+-\Psi^-}dV_{\tilde{Y}}.$$
	On the other hand we have a basic transverse K\"ahler metric $\omega_b$ on $\tilde{Y}\smallsetminus\pi^{-1}(y_0)$ such that the restriction of $\omega_b$ to $\tilde{S}$ provides us the transverse K\"ahler metric $\chi$ and by Theorem \ref{h orientable}, we have $[\chi^n]\neq 0$.
		
	Also there exist a positive constant $A$ such that $\Psi^+,\Psi^-\in PSH(\tilde{S},\xi,A\chi)$.
	Since, the discrepancy divisor $\tilde{D}$ has normal crossing and $Y$ has \textit{klt}-singularities, we can find a $q>1$ such that $-qb_j>-1$, for all $j$. This implies that $e^{-\Psi^-}\in L^q(\tilde{S},\chi^n\wedge\hat{\eta}).$
\begin{thm}[Lemma 4.6 of \cite{B20}]
	A conical smooth function $r^2$ satisfies the equation \eqref{flat} on $Y_{reg}$ if and only if $\tilde{\vp}$ satisfies the following equation on $\tilde{S}\cap\mathcal{U}$, $$\left(\tilde{\omega}+\ddbarb\tilde{\vp}\right)^n=e^{-(n+1)\tilde{\vp}}e^{(n+1)(\Psi^+-\Psi^-)}\chi^n,$$ where $\Psi^+$ and $\Psi^-$ are the functions defined above.
\end{thm}
\begin{proof}
	Let $Q=\log r^2$, then $\ddbar r^2=e^Q(\ddbar Q+\sqrt{-1}\d Q\wedge \dbar Q)$ and the equation \eqref{flat} becomes \[\label{y to s eqn} (\ddbar Q)^n\wedge\sqrt{-1}\d Q\wedge\dbar Q=e^{-(n+1)Q}(\sqrt{-1})^{(n+1)^2}\Omega\wedge\overline{\Omega}. \]
	Now let $(w=u+\sqrt{-1}v,z)$ be a holomorphic coordinate of a neighborhood in $\mathcal{U}$ such that $\xi=\frac{\d}{\d v}$ and $-J\xi=\frac{\d}{\d u}$ and define $\phi=Q-2u$, then $\xi\phi=(J\xi)\phi=0$. Hence $\phi$ is a basic function which depends only on the coordinate $z$ in $(w,z)$ and since $\ddbar u=0$ we get 
\begin{align*}
	(\ddbar Q)^n\wedge\sqrt{-1}\d Q\wedge\dbar Q=&(\ddbar\phi(z))^n\wedge\sqrt{-1}\d u\wedge\dbar u\\
		=&(\ddbar Q)^n\wedge\sqrt{-1}dw \wedge d\bar{w}\\
		=&(\ddbarb Q)^n\wedge du\wedge dv
\end{align*} 
	on $\mathcal{U}$ and also $Q=\log f_\xi+\vp$ by definition. If we use these in \eqref{y to s eqn} and contracting with the vector fields $-J\xi$ and $\xi$ we get that $$(\ddbarb Q)^n=(f_\xi)^{-(n+1)}e^{-(n+1)\vp}(\sqrt{-1})^{(n+1)^2}\Omega\wedge\overline{\Omega}(\xi,J\xi,\cdot)$$ on $\mathcal{U}.$ When we restrict the above equation to $S_{reg}$ where we have $f_\xi=1$, we get that $$(\omega_T+\ddbarb\vp)^n=e^{-(n+1)\vp}(\sqrt{-1})^{(n+1)^2}\Omega\wedge\overline{\Omega}(\xi,J\xi,\cdot).$$ If we pull back this equation using the map $\pi:\tilde{S}\to S$ and using \eqref{y to s eqn}, we get \[\label{deg eqn} (\tilde{\omega}+\ddbarb\tilde{\vp})^n=e^{-(n+1)\tilde{\vp}}e^{(n+1)(\Psi^+-\Psi^-)}\chi^n \] on $\pi^{-1}(S_{reg})$, where $\tilde{\omega}=\pi^*\omega_T$ and $\tilde{\vp}=\pi^*\vp$.
\end{proof}
	The equation \eqref{deg eqn} is a degenerate transverse complex Monge-Amp\`ere equation on the homologically orientable transverse K\"ahler manifold. Now we perturb the equation a bit to get a non-degenerate equation and we try to get the estimate for the perturbed equation using Theorem \ref{mt}. We proceed as follows.
	
	Let $\psi^-=\Psi^--\vp$. Since $\Psi^+,\psi^-\in PSH(\tilde{S},\xi,A\chi)$ for some large enough $A>0,$ by Theorem \ref{approx} there exist sequences of basic smooth functions $\psi_j^+,\psi_j^-\in PSH(\tilde{S},\xi,A\chi)$ which decreases and converges uniformly to $\Psi^+$ and $\psi^-$ respectively. 

	We let   $\omega_t:=\tilde{\omega}+t\chi$ for $t\in(1,0]$ and consider the following transverse complex Monge-Amp\`ere equation on $\tilde{S}$, which is the perturbed form of the equation \eqref{deg eqn}, \[\label{perturbed eqn} \left(\omega_t+\ddbarb\vp_{j,t}\right)^n=e^{(n+1)(\psi_j^+-\psi_j^-)}\chi^n, \] where $\vp_{j,t}$ is the solution with $\omega_t+\ddbarb\vp_{j,t}>0$ which exists by transverse Calabi-Yau theorem \cite{aziz} and $\psi_j^+$ are normalized so that $$\int_{\tilde{S}}\omega_t^n\wedge\hat{\eta}=\int_{\tilde{S}}\left(\omega_t+\ddbarb\vp_{j,t}\right)^n\wedge\hat{\eta}=\int_{\tilde{S}}e^{(n+1)(\psi_j^+-\psi_j^-)}\chi^n\wedge\hat{\eta},$$ 
	where $\hat{\eta}$ is the $1$-form defined in Theorem \ref{h orientable}. Since $\tilde{\omega}\geq 0$, we have that $\Nu_{\omega_t}=0$, then by Theorem \ref{mt}, we can find a positive constant $C_j:=C_j(n,p,\xi,\tilde{\omega},\chi,\|e^{(n+1)(\psi_j^+-\psi_j^-)}\|_{L^1(\log L)^p})$ such that $$0\leq -\vp_{j,t}\leq C_j,\text{ for all } t\in(0,1].$$ As $\psi_j^+,\psi_j^-$ decreases and converges uniformly to $\Psi^+,\psi^-\in L^q(\tilde{S},\chi^n\wedge\hat{\eta})$ and by dominated convergence theorem, there exist a constant $C:=C(n,p,\xi,\tilde{\omega},\chi,\|e^{(n+1)(\Psi^+-\psi^-)}\|_{L^1(\log L)^p})$ such that $C_j\leq C, \text{ for all } j.$ Thus \[\label{c0 est} 0\leq -\vp_{j,t}\leq C. \]
\subsection{Higher order estimates:} 
	As we have the $L^\infty$ estimate 	\eqref{c0 est} for the solutions of the perturbed equation \eqref{perturbed eqn}, in this section we try to establish the higher order estimate using the $L^\infty$ estimate. For that we need the following lemma, as it follows easily once we pick the foliated coordinate chart locally and proceed as in the case of the K\"ahler manifold, we skip the proof. Interested reader can refer to \cite[Appendix]{TTR} for the proof.
\begin{lem}[Transverse Aubin-Yau inequality]\label{tr Aubin-Yau}
	Let $\theta$ and $\vartheta$ be any two transverse K\"ahler metrics on the homologically orientable transverse K\"ahler manifold $\left(\tilde{S},\tilde{\sF},\theta\right)$ and $\ka$ be the lower bound of the transverse bi-sectional curvature of $\theta$, then
	$$\Delta_{\vartheta}\left(\log\La_{\theta}\vartheta\right)\geq -\frac{\La_{\theta}Ric^T (\vartheta)}{\La_{\theta}\vartheta}-\ka\La_{\vartheta}\theta.$$
\end{lem}
	As the second-order estimate is quite standard we provide only a sketch of the proof.
\begin{lem}\label{higher order}
	Let $K\subset\tilde{S}\cap\mathcal{U}$ be any compact set. Then there exists a positive constant $A$ depends only on $K$ such that $$|\Delta_\chi\vp_{j,t}|\leq Ae^{-\psi_j^-},$$ on $K$, where $\chi=\omega_b|_{\tilde{S}}$ is a transverse K\"ahler form on $\tilde{S}.$
\end{lem}
\begin{proof}	
	Let $\psi:=\Phi_b-\log \pi^*(f_\xi).$ Then $\psi$ is locally bounded and $\psi\to -\infty$ along the boundary of $\tilde{S}\cap\mathcal{U}.$ We define the quotient map $\pi_1:\mathcal{U}\to \tilde{S}\cap\mathcal{U}$ using the action of $-J\xi$ on $\tilde{Y}$ which also respects the transverse K\"ahler structures of both the sides, and we also have $\pi:\tilde{Y}\to Y$ a $\mathbb{T}$-equivariant resolution map. Now, let $\theta=\pi^*(\ddbarb\log f_\xi)$ a semi-positive transverse form on $\tilde{Y}\smallsetminus\pi^{-1}(y_0)$, then $(\theta+\ddbarb\psi)|_{\mathcal{U}}$ is a basic transverse K\"ahler form $\omega_b$ constructed in Theorem \ref{tr metric on sing} on $\mathcal{U}$ and $\theta|_{\tilde{S}}=\tilde{\omega}.$
	
	If we take $\psi=\Phi_b-\log \pi^*(f_\xi)$, $\omega_\vep=\theta+\ddbar\psi+\vep\omega_b$ and $\omega_\vep'=\theta+\vep\omega_b+\ddbar(\pi_1^*\vp_{j,\vep})$ over $\mathcal{U}$ for the quantities considered in Theorem B.1 of \cite[Appendix. B]{BBEGZ} then the proof follows easily by mimicking the proof step-by-step, as the calculations carried in \cite[Theorem B.1]{BBEGZ} are local in nature and we have the transverse Aubin-Yau inequality(see, Lemma \ref{tr Aubin-Yau}). However, we just add a few lines about where the proof of Theorem B.1 in \cite{BBEGZ} needed certain modifications to the foliated case to prove the estimate. 
	
	Let $\ka$ be the lower bound of the transverse bi-sectional curvature of the transverse K\"ahler metric $\chi=\omega_b|_{\tilde{S}}$ on $\tilde{S}.$ Since, $\sL_{J\xi}\omega_b=0 $, this $\ka$ will be the lower bound of the transverse bi-sectional curvature of $\omega_b$ on $(\tilde{Y}\smallsetminus\pi^{-1}(y_0),\sF).$ Then the transverse Aubin-Yau inequality holds for the forms $\omega_\vep$ and $\omega_\vep'$ on $\mathcal{U}$ also, since $\sL_{J\xi}\omega_\vep=\sL_{J\xi}\omega_\vep'=0.$ With this argument we can get a inequality similar to $(B.4)$ of \cite[Theorem B.1]{BBEGZ} for the foliated case as following \[\label{tr aub-yau ineq}\Delta_{\omega_\vep}\left(\log\La_{\omega_\vep}\omega_\vep'\right)\geq -\frac{\La_{\omega_\vep}Ric^T(\omega_\vep')}{\La_{\omega_\vep}\omega_\vep'}-A\La_{\omega_\vep'}\omega_\vep\] on $\mathcal{U}.$ As the functions on both side of the inequality are basic functions on $(\tilde{Y}\smallsetminus\pi^{-1}(y_0),\sF)$ and the codimension of the foliations $\sF$ and $\sF_\xi$ are equal to $n$, the inequality also holds true on $\tilde{S}$ with $\omega_\vep|_{\tilde{S}}$ and $\omega_\vep'|_{\tilde{S}}$, since $Ric^T(\omega_\vep')|_{\tilde{S}}=Ric^T(\omega_\vep'|_{\tilde{S}})$.
	Now let $$ \al_\vep:=\omega_\vep|_{\tilde{S}\cap\mathcal{U}}\text{ and } \al_\vep':=\omega_\vep'|_{\tilde{S}}.$$
	Then $\al_\vep=(1+\vep)\chi|_{\tilde{S}\cap\mathcal{U}}$ and $\al_\vep'=\omega_{\vep,\vp_{j,\vep}}$ and from the equation \eqref{perturbed eqn}, we have $$ -Ric^T(\al_\vep')\geq -A_1\chi-(n+1)\ddbarb\psi_j^-,$$ for some large enough constant $A_1>0.$ This with the inequality \eqref{tr aub-yau ineq} and the fact that $\sL_{J\xi}\omega_\vep=\sL_{J\xi}\omega_\vep'=0 $ implies that \[\label{la aubin-yau}\Delta_{\al_\vep}\left(\log\La_{\al_\vep}\al_\vep'\right)\geq -A_2\frac{\Delta_{\chi}\psi_j^-}{\La_{\al_\vep}\al_\vep'}-
	A_3\La_{\al_\vep'}\al_\vep\] on $\tilde{S}\cap\mathcal{U}$, where $A_2 \text{ and } A_3$ are some positive constants.
	Since $A\omega_\vep+\ddbar(\pi_1^*\psi_j^-)\geq 0$, we have \[ 0\leq An+\Delta_{\omega_\vep}(\pi_1^*\psi_j^-)\leq\left(A\La_{\omega_\vep'}(\omega_\vep)+\Delta_{\omega_\vep'}(\pi_1^*\psi_j^-)\right)\La_{\omega_\vep}(\omega_\vep') \]
	Using this in the equation \eqref{tr aub-yau ineq}, we get \[ \Delta_{\omega_\vep'}\left(\log\La_{\omega_\vep}(\omega_\vep')+\pi_1^*(\psi_j^-)\right)\geq -A\La_{\al_\vep'}(\omega_\vep) .\]
	Now as in \cite[Theorem B.1]{BBEGZ}, we also take the functions $\rho_\vep:=\pi_1^*(\vp_{j,\vep})-\psi$  and  $$H=\log\left(\La_{\omega_\vep}\omega_\vep'\right)+\pi_1^*(\psi_j^-)-A\rho_\vep,$$ on $\mathcal{U}$, where $A$ is a constant large enough and depending on $\ka.$ We also have $\omega_\vep'=\omega_\vep+\ddbar\rho_\vep$. 
	
	Since, $J\xi(\psi)=J\xi(H)=0$ and $\psi\to -\infty$ at $\d\mathcal{U}$, there exists $x\in \tilde{S}\cap\mathcal{U}$ in which $H$ attains its maximum.
	As we have the $L^\infty$-estimate \eqref{c0 est} and by proceeding as in \cite[Theore B.1]{BBEGZ}, we get the required estimate.
\end{proof}
\subsection{Proof of Corollary \ref{main thm}}
\begin{proof}[\textbf{Proof of Corollary \ref{main thm}}]
	We have the $L^\infty$ estimate \eqref{c0 est} and the Laplacian estimate in Theorem \ref{higher order} for the solutions $\vp_{j,t}$. As $\psi_j^-$ converges uniformly to $\psi^-$ and $\psi^-$ is locally bounded, on any foliated coordinate we can find constant $A>0$ such that $|\Delta_\chi\vp_{j,t}|\leq A$. Then we can find a sequence $t(j)$ such that $t(j)\to 0$ as $j\to \infty$ and $\{\vp_{j,t(j)}\}$ converges to $\tilde{\vp}_0$ in $C^1$ locally. Moreover $\tilde{\vp}_0$ satisfies the equation $$(\tilde{\omega}+\ddbarb\tilde{\vp}_0)^n=e^{-(n+1)\tilde{\vp}}e^{(n+1)(\Psi^+-\Psi^-)}\chi^n$$ on $\tilde{S}\cap\mathcal{U}$ with $\tilde{\omega}+\ddbarb\tilde{\vp}\geq 0$. Since $\tilde{\omega}|_E\equiv 0$, by maximum principle $\tilde{\vp}_0$ is constant on each irreducible component of $E$ and $\tilde{\vp}_0$ is unique up to addition of a constant. So, $\tilde{\vp}_0=\tilde{\vp}+c$, for some constant $c.$
	Since $$\tilde{\vp}\in L^\infty(\tilde{S}) \text{ and } \Delta_\chi\tilde{\vp}\in L^\infty_{loc}(\tilde{S}),$$ by Evans-Krylov theory we get the $C^{2,\al}$ bound and by Schauder estimate we get the higher order estimates. Hence, $\tilde{\vp}$ is smooth.
	
	Thus, the function $\tilde{\vp}$ can be pushed forward as a function $\vp$ on $S$ and it is smooth on $S_{reg}$, this implies that $r^2$ is smooth on $Y_{reg}$, since $r^2=e^\vp f_\xi$.
	
\end{proof}

\subsection*{Acknowledgements:}
	The author would like to thank his PhD. advisor Ved V. Datar for suggesting this problem, for the helpful discussions and the suggestions for improving the first draft of this work. This research was supported by NBHM fellowship and the graduate program of the Indian Institute of Science.

\end{document}